\documentclass[preprint,hidelinks,onefignum,onetabnum]{siamart220329}
\usepackage{xcolor}

\usepackage{lipsum}
\usepackage{amsfonts}
\usepackage{graphicx}
\usepackage{epstopdf}
\usepackage{algorithmic}
\ifpdf
  \DeclareGraphicsExtensions{.eps,.pdf,.png,.jpg}
\else
  \DeclareGraphicsExtensions{.eps}
\fi

\newsiamremark{remark}{Remark}
\newsiamremark{hypothesis}{Hypothesis}
\crefname{hypothesis}{Hypothesis}{Hypotheses}
\newsiamthm{claim}{Claim}

\headers{GAP Method for Kinetic Equations in the Diffusive Limit}{G. Ceruti, N. Crouseilles, L. Einkemmer}

\title{A Galerkin Alternating Projection Method for Kinetic Equations in the Diffusive Limit 
\thanks{Submitted to the editors May 27, 2025.}
}

\author{
	Gianluca Ceruti\thanks{ University of Innsbruck (\email{Gianluca.Ceruti@uibk.ac.at}, \email{Lukas.Einkemmer@uibk.ac.at}).}
	\and Nicolas Crouseilles \thanks{Univ Rennes, Inria  (MINGuS),   CNRS,   IRMAR (UMR   6625) \& ENS Rennes,   France   (\email{nicolas.crouseilles@inria.fr}.}
	\and Lukas Einkemmer\footnotemark[2]}

\usepackage{amsopn}

\ifpdf
\hypersetup{
  pdftitle={GAP method for kinetic equations in the diffusive limit},
  pdfauthor={G. Ceruti, N Crouseilles, L Einkemmer}
}
\fi

\usepackage{graphicx} 
\usepackage{enumitem}
\usepackage{amssymb}
\usepackage{xcolor}
\usepackage{cancel}
\usepackage{lipsum}

\newcommand{\R}{{\mathbb R}}
\newcommand{\M}{{\mathcal M}}

\begin{document}

\maketitle

\begin{abstract}
The numerical approximation of high-dimensional evolution equations poses significant computational challenges, particularly in kinetic theory and radiative transfer. In this work, we introduce the Galerkin Alternating Projection (GAP) scheme, a novel integrator derived within the Dynamical Low-Rank Approximation (DLRA) framework. 
We perform a rigorous error analysis, establishing local and global accuracy using standard ODE techniques. Furthermore, we prove that GAP possesses the Asymptotic-Preserving (AP) property when applied to the Radiative Transfer Equation (RTE), ensuring consistent behavior across both kinetic and diffusive regimes. In the diffusive regime, the K-step of the GAP integrator directly becomes the limit equation. In particular, this means that we can easily obtain schemes that even in the diffusive regime are free of a CFL condition, do not require well prepared initial data, and can have arbitrary order in the diffusive limit (in contrast to the semi-implicit and implicit schemes available in the literature). Numerical experiments support the theoretical findings and demonstrate the robustness and efficiency of the proposed method.
\end{abstract}

\begin{keywords}
Kinetic equations, dynamical low-rank approximation, asymptotic-preserving methods, radiative transfer equation
\end{keywords}

\begin{MSCcodes}
  	65M70, 
	35Q20, 
	82C40, 
	65M12, 
	65L05, 
	65F30  
\end{MSCcodes}

\section{Introduction}

Kinetic equations play a central role in modeling the evolution of particle distributions in rarefied gases, radiative transfer, and plasma physics. These equations are inherently high-dimensional, often involving complex interactions between spatial and velocity/angle variables. A promising direction to mitigate these difficulties is the use of model reduction techniques. Among them, Dynamical Low-Rank Approximation (DLRA)~\cite{koch2007dynamical} has emerged as a powerful tool for reducing computational complexity by approximating the solution on a low-dimensional manifold. Several DLRA-based integrators have been proposed in recent years, including the projector-splitting integrator (PSI)~\cite{lubich2014projector, einkemmer2018low} and the Basis-Update Galerkin (BUG)~\cite{ceruti2022unconventional,ceruti2022rank,ceruti2024parallel,ceruti2024robust} scheme.

The presence of (possibly) small parameters, such as the Knudsen number, further leads to stiff regimes where traditional numerical methods become prohibitively expensive or unstable. In particular, capturing the correct asymptotic behavior as the system transitions from the kinetic to the diffusive regimes is a major computational challenge. Such methods are called \textit{asymptotic preserving} (AP). There are a number of works in the literature that consider dynamical low-rank AP schemes. In \cite{Ding2021} the authors considers a PSI based DLRA scheme for the radiative transfer equation, where all substeps are treated with a fully implicit scheme. It was rigorously shown in that work that the resulting scheme is asymptotic preserving in the diffusive limit. Fully implicit methods, however, can be expensive, and to overcome this a micro-macro decomposition based IMEX scheme was proposed in \cite{einkemmer2021asymptotic}. An added benefit of this is that the negative time step in the PSI is not an issue as the DLRA is only applied to the micro part. The problem of the negative time step can also be overcome by using the BUG integrator (which simply does not need a negative step). In \cite{einkemmer2024asymptotic} an IMEX scheme for the BUG integrator has been proposed that is asympotic preserving and energy stable. Despite being semi-implicit the IMEX schemes in \cite{einkemmer2021asymptotic} and \cite{einkemmer2024asymptotic} still need to satisfy a hyperbolic CFL condition in the kinetic regime and a parabolic CFL condition in the diffusive regime. Also note that even for fully implicit schemes there is no guarantee that we are unconditionally stable in the diffusive limit. Later asymptotic preserving DLRA schemes have also been developed for a range of other kinetic problems; see, e.g., \cite{einkemmer2021efficient} for the Boltzmann--BGK equation, \cite{Coughlin2022} for the Vlasov-Ampere-Fokker-Planck equation, \cite{Patwardhan2024} for the thermal radiative transfer equation, and \cite{einkemmer2025asymptotic} for the full Boltzmann equation. We also mention the work \cite{sands2024high} which considers a AP low-rank scheme in the step-and-truncate framework.

In this paper, we propose a robust dynamical low-rank integrator that consists of two steps that update the angle and spatial dependent low-rank factors, respectively. While this integrator can, in principle, be applied to arbitrary evolution equations, we focus here on kinetic equations with a collisional limit (for simplicity we will consider a diffusive limit). In this limit, the angle update performs a projection onto the appropriate angle subspace (thus, for well prepared initial data nothing is done). Moreover, the update of the spatial low-rank factors converges to the appropriate limit equation of the moments. Thus, in this limit, the scheme not only is asymptotic preserving, but the low-rank factors can be interpreted directly as the moments of the solution. We also note that the integrator is less computational demanding compared to both micro-macro DLRA approaches as no separate macro step is needed and also requires less code to implement than classic DLRA integrators (such as the PSI or the BUG integrator). We emphasize, in particular, that our scheme does not need to satisfy a CFL condition. This is in stark contrast to most schemes in the literature, where even for semi-implicit schemes a hyperbolic CFL condition in the kinetic regime and, more severe, a parabolic CFL condition in the diffusive regime needs to be satisfied.

The proposed integrator has similarities to other robust DLRA integrators proposed in the literature. In fact, it uses similar substeps compared to both the projector splitting and BUG integrator and similar to the BUG integrator it can be interpreted as a predict then Galerkin scheme. The difference is that the prediction here is performed only for the angle dependent low-rank factors and the Galerkin step is performed for the spatial low-rank factors (not only for the coefficients as is the case in BUG). A similar idea has been followed in deriving the XL integrator \cite{einkemmer2025asymptotic}. However, in this case, the collision is split from the advection and the XL integrator is then only applied to the collision. Since the collision operator only acts in velocity space, this simplifies the method. In contrast, our scheme applied to the full equation directly results in low-rank factors that can be interpreted as moments in the relevant limit. We also note that no convergence analysis is performed in \cite{einkemmer2025asymptotic}; however, the collision operator considered is significantly more complicated (full Boltzmann operator) than what we consider in this paper. In parallel, a similar type of integrator has been introduced in the context of gradient flow problems~\cite{kusch2025augmented}, where an augmented backward-corrected PSI (abcPSI) scheme was proposed as a variant of the standard projector-splitting integrator. While that work provides an effective framework for DLRA-based optimization, its setting and formulation differ substantially from the PDE-based perspective adopted here. Building on the approach initially proposed in~\cite{einkemmer2018low,bachmayr2021existence}, the present work introduces a unifying PDE-oriented formulation and asymptotic analysis framework for the proposed integrator.


We first introduce the new integrator and compare it with classical dynamical low-rank (DLRA) algorithms (Section~2). Then, in Section~3, we perform a rigorous convergence analysis for a generic problem, demonstrating that the proposed integrator exhibits error behavior comparable to that of classical DLRA schemes. In Section~4, we show that, when applied to the scaled RTE, the integrator is asymptotic-preserving and that the leading spatially dependent low-rank factor can be interpreted as the particle density. Section~5 discusses the discretization in space, angle, and time, while Section~6 presents numerical results that support and validate the theoretical findings. 

Throughout the text, \textbf{vector} and \textbf{matrix} quantities are denoted in bold, while scalar quantities are written in standard (non-bold) font.

\section{The Galerkin alternating projection scheme}

The numerical approximation of high-dimensional evolution equations is a fundamental challenge in computational mathematics, often represented by the abstract problem of determining the solution $f = f(t, \xi) \in \mathbb R$  of 
\begin{equation} \label{eq:fullEq}
	\partial_t f = \mathcal{L}(f), \quad f(t_0) = f_0 \in L^2(\Omega,\omega)\, ,
\end{equation}
where $\xi \in \Omega\subseteq \R^d$, and the function space $L^2(\Omega, \omega)$ consists of all functions that are square-integrable with respect to the measure $\omega$. Here, $\Omega$ may represent either a countable set (discrete case) or a subset of $\mathbb{R}^d$ with positive Lebesgue measure (continuous case). This allows us to treat the continuous case, where only a low-rank approximation is performed, and the discrete case, where in addition the problem is discretized in space and velocity, within the same general mathematical framework.
For clarity, we recall that given a measure $\omega$ on a set $\Omega$, the associated $\omega$-scalar product and the corresponding induced norm are defined as 
\[ 
    \langle f, g \rangle_{\omega}  := \int_{\Omega} f g \ d\omega, 
    \quad 
    \| f \|_{\omega} := \sqrt{\langle f, f\rangle_{\omega}} \, .
\] 
In this work, we focus on numerical approximations of \eqref{eq:fullEq}, with particular emphasis on the radiation transport equation (RTE). We introduce a novel dynamical low-rank integrator, called the GAP scheme, as an efficient method for solving problems arising in applications related to kinetic theory.

The GAP scheme belongs to the class of robust numerical integrators recently developed in the context of Dynamical Low-Rank Approximation (DLRA), an online model reduction technique for high-dimensional problems. To provide a self-contained discussion, we introduce key elements of DLRA in the subsequent section for readers unfamiliar with this framework.

\subsection{Dynamical low-rank approximation}
The numerical approximation of the abstract time-dependent differential equation \eqref{eq:fullEq} often encounters the curse of dimensionality, i.e.~an exponential growth in the number of parameters or degrees of freedom required to achieve an accurate numerical approximation of the analytical solution. Note that the latter is rarely available which necessitates appropriate numerical schemes. In this context, time-dependent model order reduction techniques are highly appealing.

Within this framework, we review the main elements and recent progress of a variational approach known as Dynamical Low-Rank Approximation (DLRA). The core idea of DLRA is to constrain the dynamics of \eqref{eq:fullEq} to an ansatz space $\mathcal M$, for which the numerical approximation benefits from reduced storage and computational complexity.

For notational simplicity, we focus on the two-variable case \( f=~f(t, \xi_1, \xi_2)\),  which we assume admits the following factorization: 
\begin{equation} \label{eq:lowRankf}
    f(t,\xi_1, \xi_2) =  \sum_{i,j}^r S_{ij}(t) X_i(t,\xi_1) V_j(t,\xi_2) \in \R \, .
\end{equation}
This example serves to illustrate the method, which extends naturally to functions of arbitrary dimension. In this factorization, the coefficient matrix \( S_{ij}(t) \in \mathbb{R}^{r \times r} \) captures the time-dependent interactions between the basis functions \( X_i(t, \xi_1) \) and \( V_j(t, \xi_2) \). The functions \( X_i \) and \( V_j \) form a time-evolving reduced basis that must be determined as part of the decomposition. 

The minimal number $r$ used in \eqref{eq:lowRankf} is referred to as the \emph{rank} of the function $f \in L^2(\Omega, \omega)$. The corresponding ansatz set $\mathcal{M}_r^\omega$ consists of $L^2(\Omega, \omega)$-functions that admit a rank-$r$ representation and is defined as 
\begin{equation}
    \M_r^\omega := \{ f \in L^2(\Omega, \omega) \ | \ \text{rank}(f) = r \} \, .
\end{equation}
For reasons of brevity, we will omit the measure in the definition of the ansatz set if it is clear from the surrounding context. To summarize the key results available in the literature, we require the following assumptions:
\begin{enumerate}[label=(\alph*)]
    \item \emph{Product measure on a (compact) Cartesian product set:} 
        \[\omega = \omega_1 \cdot \omega_2 \quad \text{where} \quad \Omega = \Omega_1 \times \Omega_2 \, .\]
    \item \emph{Bounded and Lipschitz operator:} 
        \[ \| \mathcal{L}(f) \|_\omega \leq B \, , \quad
         \| \mathcal{L}(f) - \mathcal{L}(g)  \|_\omega \leq C \|f-g\|_\omega \, . \]
\end{enumerate}
The transition from global to local assumptions follows naturally; for simplicity, we adopt the global ones. Moreover, while assumption (b) is typically not satisfied by general operators, its discretized counterpart generally exhibits these properties.
While \eqref{eq:lowRankf} it is commonly used, for what follows it is convenient to reformulate the problem in linear algebra notation. This can be achieved by introducing the vector notation
\begin{align}  
  \label{vectorX}
  \mathbf X(t, \xi_1) &:= [ X_1(t, \xi_1) \ | \ X_2(t, \xi_1) \ | \ \dots \ | \ X_r(t, \xi_1)] \in \R^{1 \times r}, \\
  \label{vectorV}
   \mathbf V(t, \xi_2) &:= [ V_1(t, \xi_2) \ | \ V_2(t, \xi_2) \ | \ \dots \ | \ V_r(t, \xi_2)] \in \R^{1 \times r} \, .
\end{align}
Note that each element of these vectors are still functions in $t$ and either $\xi_1 \in \Omega_1$  or $\xi_2 \in \Omega_2 $. Thus, we obtain the matrix representation $f = \mathbf X \mathbf S \mathbf V^\top$,  which resembles a low-rank singular value decomposition. Indeed, If the sets $\Omega_1$ and $\Omega_2$ are \emph{discrete}, the ansatz set $\mathcal{M}_r$ reduces to the space of rank-$r$ matrices, which is the setting where DLRA was originally formulated in. As shown in \cite{koch2007dynamical}, $\mathcal{M}_r$ possesses a manifold structure, thus admitting local tangent spaces where an orthogonal projector can be defined. 

We now recall the construction of the orthogonal projector onto the tangent space. Let $f=\mathbf{XSV}^\top \in \mathcal{M}_r$, we introduce the intermediate projectors:
\begin{align}
    &\mathcal{P}_\mathbf{X}[\cdot] = \sum_{i=1}^r X_i \langle X_i, \cdot \rangle_{\omega_1}  = \mathbf{X} \langle \mathbf{X}^\top \ \cdot \rangle_{\omega_1}\, , \\
    &\mathcal{P}_\mathbf{V}[\cdot] = \sum_{j=1}^r \langle \cdot, V_j \rangle_{\omega_2 } V_j = \langle \cdot \ \mathbf{V} \rangle_{\omega_2} \mathbf{V}^\top \, ,
\end{align}
where, for a given function $g$ and measure $\omega$, we set $\langle g \rangle_\omega := \langle g, 1 \rangle_\omega$. When $g$ is vector-valued, all operations above are understood component-wise.
The orthogonal projection $\mathcal P(f)[\cdot]$ onto the tangent space $\mathcal{T}_f \mathcal{M}_r$ is then obtained as
\begin{equation}
\label{def:proj_tangent}
    \mathcal P(f)[\cdot] = \mathcal P_\mathbf{X}[\cdot] - \mathcal P_\mathbf{X} \circ \mathcal P_\mathbf{V}[\cdot]  + \mathcal P_\mathbf{V}[\cdot] \in \mathcal T_f \M_r  \, .
\end{equation}
Finally, the DLRA approximation of \eqref{eq:fullEq}, denoted by $y$, is obtained by constraining the dynamics to $\mathcal{M}_r$ through the projection of the right-hand side onto the tangent space $\mathcal{T}_f \mathcal{M}_r$, i.e.
\begin{equation} \label{eq:DLReq}
    \partial_t y = \mathcal P(y)  [ \mathcal L(y)], 
    \quad 
    y(t_0) = \Pi_{\M_r}(f_0) \in \M_r \, .
\end{equation}
where $\Pi_{\M_r}(f_0)$ denotes the best approximation of $f_0$ into $\M_r$. Before proceeding further, we observe that by construction the DLRA solution $y$ belongs to $\mathcal M_r$ and the following  holds 
\[
	\mathcal{L}(y) = \mathcal{P}(y)[\mathcal{L}(y)] + \mathcal{P}^\perp(y)[\mathcal{L}(y)] \, , 
\]
where $\mathcal P^\perp(y)$ denotes the orthogonal complement. To ensure that the DLRA approximation remains close to the full-order model \eqref{eq:fullEq}, a model-error assumption on the orthogonal complement must be introduced. Therefore, following the DLRA literature, we complement assumptions (a) and (b) with the following assumptions:
\begin{enumerate}[label=(\alph*), start=3]
    \item \emph{Model order assumption:} The projection of \( \mathcal{L}(f) \) onto the orthogonal complement of the tangent space is small, ensuring the DLRA approximation remains close to the full-order model:  
        \[ \| \mathcal P^\perp(f)[\mathcal{L}(f)] \|_\omega \leq \eta \qquad \forall f \in \mathcal M_r^\omega \, .\] 
    \item \emph{Small initial value error:} The initial value error 
    \[ \delta_0 := \| f_0 - \Pi_{\M_r}(f_0) \|_\omega\] for the DLRA evolution equation \eqref{eq:DLReq} is assumed to be small.
\end{enumerate}


\subsection{Projector-splitting integrator}
While DLRA provides an effective framework for reducing computational complexity, the projection onto the tangent space can exhibit large local Lipschitz constants, which scale with the local curvature of the manifold. In practical computations this can usually not be avoided because we often want to run the algorithm in a regime where at least one singular value of $\bf{S}$ is small. This sensitivity poses numerical challenges, necessitating the development of robust integrators that remain stable regardless of the curvature of the ansatz set \( \mathcal{M}_r \).  

To address this issue, significant efforts have been devoted to designing and analyzing numerical schemes that circumvent these difficulties. One of the first and most widely studied approaches is the Projector-Splitting Integrator (PSI). This method avoids direct projection onto the tangent space at each step and instead decomposes the evolution into substeps that operate on individual components of the low-rank factorization. 
The projector splitting integrator is inspired by a natural decomposition of the projection operator onto the tangent space, as given in equation \eqref{def:proj_tangent}.

One timestep from $t=t_0$ to $t=t_0+\Delta t$ of the PSI integrator with initial value $f(t_0) \in \mathcal M_r^\omega$ can be formulated as follows:  

\begin{algorithm}[H]
    \caption{The Projector-Splitting Integrator scheme (PSI)}
    \begin{algorithmic}
        \STATE{\textbf{Input:} $f(t_0) = \mathbf{X}_0 \mathbf{S}_0 \mathbf{V}_0^\top \in \mathcal M_r^\omega \, .$}
        
        \STATE{}
        \STATE{
            \qquad Solve from time $t=t_0$ until $t = t_0 + \Delta t$ 
            $$\partial_t \ell = \mathcal P_{\mathbf X_0} [\mathcal L(\ell)], \quad \ell(t_0) =  f(t_0)\, .$$
            \qquad Set $\mathbf V_1$ to be an orthonomal basis of $\ell(t_0 + \Delta t)$ w.r.t. to $\omega_2$.
            \\ \ \\
            }
        \STATE{
            \qquad Solve from time $t=t_0$ until $t = t_0 + \Delta t$ 
            $$\partial_t s = -\mathcal  P_{\mathbf X_0} \circ P_{ \mathbf V_1}  [\mathcal L(s)], \quad s(t_0) = \ell(t_0 + \Delta t) \, .$$ 
            }
        \STATE{
            \qquad Solve from time $t=t_0$ until $t = t_0 + \Delta t$ 
            $$\partial_t k = \mathcal P_{\mathbf V_1} [ \mathcal L(k)], \quad k(t_0) = s(t_0 + \Delta t) \, .$$
            \qquad Set $\mathbf X_1$ to be an orthonomal basis of $k(t_0 + \Delta t)$ w.r.t. to $\omega_1$. 
            \\ \ \\
            }
        
        \RETURN{$k(t_0 + \Delta t) = \mathbf X_1 \mathbf S_1 \mathbf V_1^\top$ where $ \mathbf S_1 = \langle \mathbf X_1^\top  k(t_0 + \Delta t)   \mathbf V_1\rangle_\omega$ . 	
        }
    \end{algorithmic}
    \label{algoPSI}
\end{algorithm}


Before proceeding further, we clarify the meaning of the orthonormal basis set appearing at the end of steps one and three of the algorithm. To establish this interpretation and further exploit analogies with the classical formulation of the PSI, as presented in~\cite{lubich2014projector}, we observe that the following results hold.

\begin{lemma}
    The solution of the $\ell$-step of the PSI algorithm  remains in the span of $\mathbf{X_0}$. 
\end{lemma}

\begin{proof}
    We observe that
    \[
        \ell(t) 
            = f(t_0) + \int_{t_0}^{t} \partial_t \ell 
            = \mathcal P_{\mathbf X_0}[f(t_0)] + \int_{t_0}^{t} \mathcal P_{\mathbf X_0}[ \mathcal L(\ell)] 
            = \mathcal P_{\mathbf X_0}\Big[ f(t_0) + \int_{t_0}^{t} \mathcal L(\ell) \Big] \, .
    \]
    Thus, we obtain that the solution at time $t$ of the $\ell$-step lies in the span of $ \mathbf {X_0}$, i.e.
    \[
        \ell(t) = \mathcal P_{\mathbf X_0}[ \ell(t)] \, .
    \]
\end{proof}

By an analogous proof, we obtain a similar result for the \( s \)-step, where the solution remains in the span of \( \mathbf X_1 \) and \( \mathbf V_0 \), and for the \( k \)-step, where the solution remains in the span of \( \mathbf V_1 \). This allows us to make two key observations. First, due to the preceding lemma, the standard formulation of the projector-splitting integrator coincides with the one provided above, i.e.,
\[
    \partial_t \ell 
        = \mathcal P_{\mathbf X_0}\Big[ \mathcal L \big( \mathcal P_{\mathbf X_0}[\ell] \big) \Big] 
        = \mathcal P_{\mathbf X_0}[ \mathcal L\big( \ell \big) ] \, .
\]
The same holds for the the $s$-step and $k$-step. While the first equality provides the formulation used in the original PSI algorithm and is more precise and effective for obtaining the time-evolving factors, we maintain the compact representation from the last equation throughout the manuscript, as it is more convenient for analysis. 

Second, we observe that, due to the preceding lemma, there exists a set of \( r \) basis functions, which we denote by the vector \( \mathbf L = \mathbf L(t, \xi_2) \in \mathbb{R}^{1 \times r} \), such that  
\[
	\ell(t_0 + \Delta t) = \mathbf X_0 \mathbf L^\top\, .
\]  
Thus, a set of \( r \) basis functions orthonormal with respect to the measure \( \omega_2 \) can be obtained, for example, using a Gram-Schmidt orthonormalization algorithm applied to the basis functions derived from the vector \( \mathbf L \).

\subsection{The novel GAP scheme}

The PSI method has marked a significant milestone in the development of robust numerical integrators for DLRA, its effectiveness has been supported by a successful error analysis in the discrete setting, demonstrating independence from the geometry of the low-rank manifold~\cite{kieri2016discretized}, as well as extensive numerical evidence of its accuracy and efficiency in complex dynamical problems. However, its backward intermediate sub-step has the drawback of inducing additional instabilities in dissipative and collisional  problems \cite{einkemmer2025asymptotic}, even in relatively simple cases such as the heat equation. To address this issue, a novel class of algorithms, referred to as Basis and Update Galerkin (BUG), has been recently introduced in the literature.

The BUG numerical integrator was the first to eliminate the backward instability and is derived as a variant of the PSI scheme. For completeness, we recall the scheme as originally presented in~\cite{lubich2014projector}. A single time-integration step proceeds as follows:
 \begin{algorithm}[H]
 \label{algoBUG}
     \caption{The Basis and Update Galerkin scheme (BUG)}
     \begin{algorithmic}
     	\STATE{\textbf{Input:} $f(t_0) = \mathbf X_0 \mathbf S_0 \mathbf V_0^\top \in \mathcal M_r^\omega \, .$}
         \STATE{}
         \STATE{
             \qquad Solve from time $t=t_0$ until $t = t_0 + \Delta t$ 
             $$\partial_t \ell = \mathcal P_{\mathbf X_0} [\mathcal L(\ell)], \quad \ell(t_0) = f(t_0) \, .$$
             \qquad Set $\mathbf V_1$ to be an orthonomal basis of $\ell(t_0 + \Delta t)$ w.r.t. to $\omega_2$.
             \\ \ \\
             }  
         \STATE{
             \qquad Solve from time $t=t_0$ until $t = t_0 + \Delta t$ 
             $$\partial_t k = \mathcal P_{\mathbf V_0} [ \mathcal L(k)], \quad k(t_0) = f(t_0) \, .$$
             \qquad Set $\mathbf X_1$ to be an orthonomal basis of $k(t_0 + \Delta t)$ w.r.t. to $\omega_1$. 
             \\ \ \\
             }   
         \STATE{
             \qquad Solve from time $t=t_0$ until $t = t_0 + \Delta t$ 
             $$\partial_t s = \mathcal  \mathcal P_{\mathbf X_1} \circ \mathcal P_{\mathbf V_1}  [\mathcal L(s)], \quad s(t_0) = \mathcal P_{\mathbf X_1} \circ \mathcal P_{\mathbf V_1} [f(t_0)] \, .$$ 
             }
         \RETURN{$s(t_0 + \Delta t) = \mathbf X_1 \mathbf S_1 \mathbf V_1^\top$ where $ \mathbf S_1 = \langle \mathbf X_1^\top s(t_0 + \Delta t) \mathbf V_1\rangle_\omega$ . 
         }
     \end{algorithmic}
 \end{algorithm}

The BUG integrator eliminates the backward sub-step of the PSI scheme and introduces greater parallelism. Additionally, we note that any potential errors arising from an incorrect initial value at the last step, where \( \mathcal P_{X_1} \circ \mathcal{P}_{V_1}[f(t_0)] \) may differ from \( f(t_0) \), can be corrected through augmentation strategies~\cite{ceruti2022rank}.  

However, while the PSI scheme is derived from a projector-splitting approach, which induces favorable properties such as the construction of high-order schemes via composition~\cite{lubich2014projector,hairer2006geometric} and the preservation of physical invariants for the Schrödinger or Vlasov-Poisson equation~\cite{lubich2015time,einkemmer2019quasi}, the BUG integrator is an artificially constructed variant with no a priori well-defined structure. This makes its analysis and theoretical properties, though available, less immediately accessible.

In this manuscript, we focus exclusively on the fixed-rank setting and introduce a novel algorithm that combines and balances the strengths and limitations of both the BUG and PSI integrators. Therefore, we present the following novel GAP scheme:

\begin{algorithm}[H]
\label{algoGAP}
    \caption{The Galerkin Alternating Projection scheme (GAP)}
    
    \begin{algorithmic}
        \STATE{\textbf{Input:} $f(t_0) = \mathbf X_0 \mathbf S_0 \mathbf V_0^\top \in \mathcal M_r^\omega \, .$}
        \STATE{}
        \STATE{
            \qquad Solve from time $t=t_0$ until $t = t_0 + \Delta t$ 
            $$\partial_t \ell = \mathcal P_{\mathbf X_0} [\mathcal L(\ell)], \quad \ell(t_0) = f(t_0) \, .$$
            \qquad Set $\mathbf V_1$ to be an orthonomal basis of $\ell(t_0 + \Delta t)$ w.r.t. to $\omega_2$.
            \\ \ \\
            }    
        \STATE{
            \qquad Solve from time $t=t_0$ until $t = t_0 + \Delta t$ 
            $$\partial_t k = \mathcal P_{\mathbf V_1} [ \mathcal L(k)], \quad k(t_0) = \mathcal P_{\mathbf V_1}[f(t_0)] \, .$$
            \qquad Set $\mathbf X_1$ to be an orthonomal basis of $k(t_0 + \Delta t)$ w.r.t. to $\omega_1$. 
            \\ \ \\
            }
       	\RETURN{$k(t_0 + \Delta t) = \mathbf X_1 \mathbf S_1 \mathbf V_1^\top$ where $ \mathbf S_1 = \langle \mathbf X_1^\top k(t_0 + \Delta t) \mathbf V_1\rangle_\omega$ . 
    }
    \end{algorithmic}
\end{algorithm}


\section{Error analysis of the GAP scheme}
This section is devoted to the analysis of the time accuracy of the proposed GAP algorithm. Since these schemes are inherently time-marching methods, we temporarily omit the influence of additional discretization effects and focus solely on their time accuracy. We begin by recalling the local error of the projector-splitting integrator. 
To fix notation, we denote by \( y^{\mathrm{PSI}}(t_0 + \Delta t) \) the numerical solution obtained by applying Algorithm~\ref{algoPSI} to the DLRA equation~\eqref{eq:DLReq} on the interval \([t_0, t_0 + \Delta t]\), and similarly by \( y^{\mathrm{GAP}}(t_0 + \Delta t) \) the solution computed with Algorithm~\ref{algoGAP}. The quantities \( \eta \) and \( \delta_0 \) appearing in the error estimates are those defined in the DLRA assumptions~(c) and~(d), respectively.

 \begin{lemma}[PSI Local error {\cite{kieri2016discretized}}] \label{lemma:PSIerror}
	The local error of the PSI approximation at time \( t_1 = t_0 + \Delta t \) satisfies
    \[
        \|  y^{\footnotesize \rm{PSI}}(t_1) - f(t_1) \| \lesssim \Delta t^2 + \eta \Delta t  + \delta_0 \Delta t\, .
    \]
\end{lemma}
A key step in the analysis is to control the time variation of the projection error. This is achieved by exploiting the Lipschitz continuity of the exact solution, as established in the following lemma.

\begin{lemma} \label{lemma:projErr}
	Let \( \Phi_{\mathbf{V}}(t) := \| f(t) - \mathcal{P}_{\mathbf{V}}[f(t)] \| \), where \( \mathcal{P}_{\mathbf{V}} \) denotes the orthogonal projector onto a fixed subspace \( \mathbf{V} \), and \( f \) is the solution of \eqref{eq:fullEq}. Then \( \Phi_{\mathbf{V}} \) is Lipschitz continuous on \( [t_0, t_1] \), and the following bound holds
	\begin{equation} \label{eq:projLipError}
		\Phi_{\mathbf{V}}(s) \leq \Phi_{\mathbf{V}}(t_1) + 2B |t_1 - s| \qquad \forall s \in [t_0, t_1] \,,
	\end{equation}
	where \( B \) is the bound on the operator \( \mathcal{L} \) as specified in DLRA assumption~(b).
\end{lemma}

\begin{proof}
We estimate the Lipschitz regularity of \( \Phi_{\mathbf{V}} \) by rewriting it in terms of the orthogonal projector and applying the reverse triangle inequality:
	\begin{align*}
		|\Phi_{\mathbf{V}}(t) - \Phi_{\mathbf{V}}(s)|
		&= \left| \| \mathcal{P}_{\mathbf{V}}^\perp f(t) \| - \| \mathcal{P}_{\mathbf{V}}^\perp f(s) \| \right| \\
		&\leq \| \mathcal{P}_{\mathbf{V}}^\perp f(t) - \mathcal{P}_{\mathbf{V}}^\perp f(s) \| \\
		&= \| \mathcal{P}_{\mathbf{V}}^\perp (f(t) - f(s)) \| \\
		&\leq \| f(t) - f(s) \|.
	\end{align*}
	Since \( \partial_t f = \mathcal{L}(f) \) and \( \mathcal{L} \) is bounded, we have
	\[
	\| f(t) - f(s) \| \leq \int_s^t \| \partial_t f(\tau) \|\, d\tau = \int_s^t \| \mathcal{L}(f(\tau)) \|\, d\tau \leq B |t - s|,
	\]
	which implies
	\[
	| \Phi_{\mathbf{V}}(t) - \Phi_{\mathbf{V}}(s) | \leq B |t - s|.
	\]
	Hence, \( \Phi_{\mathbf{V}} \) is Lipschitz continuous on \( [t_0, t_1] \). Next, we estimate
	\begin{align*}
		\Phi_{\mathbf{V}}(s) 
		&= \| f(s) - \mathcal{P}_{\mathbf{V}}[f(s)] \| \\
		&= \| [f(s) - f(t_1)] + [f(t_1) - \mathcal{P}_{\mathbf{V}}[f(t_1)]] + [\mathcal{P}_{\mathbf{V}}[f(t_1)] - \mathcal{P}_{\mathbf{V}}[f(s)]] \| \\
		&\leq \| f(s) - f(t_1) \| + \| \mathcal{P}_{\mathbf{V}}[f(t_1)] - \mathcal{P}_{\mathbf{V}}[f(s)] \| + \Phi_{\mathbf{V}}(t_1) \\
		&\leq 2 \| f(s) - f(t_1) \| + \Phi_{\mathbf{V}}(t_1) \\
		&\leq 2 B |s - t_1| + \Phi_{\mathbf{V}}(t_1) \, .
	\end{align*}
	This concludes the proof of the Lipschitz regularity of \( \Phi_{\mathbf{V}} \).
\end{proof}

Based on the aforementioned lemmas, we derive the temporal local error analysis of the novel GAP scheme.

\begin{theorem}[GAP Local error]
   The local error of the novel GAP-algorithm  at time \( t_1 = t_0 + \Delta t \) satisfies
    \[
        \|   y^{\footnotesize \rm{GAP}}(t_1) - f(t_1)\| \lesssim \Delta t^2 + \eta \Delta t + \delta_0 \Delta t \, .
    \]
\end{theorem}

\begin{proof} 
	Let \( \mathbf{V}_1 \) denote the right singular subspace generated in the first step of the GAP algorithm. We begin by considering the triangle inequality
    \[
        \| y^{\footnotesize \rm{GAP}}(t_1) - f(t_1)\| 
        \leq \| y^{\footnotesize \rm{GAP}}(t_1) -  \mathcal P_{\mathbf V_1}[f(t_1)] \| + \| \mathcal P_{\mathbf V_1}[f(t_1)] - f(t_1) \| \, .
    \]
	The second term is bounded using the sub-optimality of the PSI approximation, i.e.,
    \begin{equation} \label{eq:subOptPSI}
        \| \mathcal P_{\mathbf V_1}[f(t_1)] - f(t_1)\|
        \leq \|    \mathcal P_{\mathbf V_1}[ y^{\footnotesize \rm{PSI}}(t_1)] - f(t_1)  \| \lesssim \Delta t^2 + \eta \Delta t + \delta_0 \Delta t \, .
    \end{equation}
 	This estimate follows from the fact that both PSI and GAP share the same initialization step and thus generate the same subspace \( \mathbf{V}_1 \), so that \( \mathcal{P}_{\mathbf{V}_1}[y^{\rm PSI}(t_1)] = y^{\rm PSI}(t_1) \). Therefore, the sub-optimality bound is a direct consequence of the local approximation error of the PSI method, as stated in Lemma~\ref{lemma:PSIerror}. Next, we introduce the intermediate variable
    \[
            \kappa(t) := \mathcal P_{\mathbf V_1} [f(t)] \, .
    \]
    Up to a defect term, we can rewrite
    \[
        \mathcal L(f)= \mathcal L(\kappa) + d \, ,
    \]
    where $d := \mathcal L(f) - \mathcal L(\kappa)$. Due to the DLRA assumption~(b) on the operator $\mathcal L$, the defect is bounded by
    \[
        \| d(t) \| 
            = \| \mathcal L(f(t)) - \mathcal L(\kappa(t)) \| 
            \leq C \| f(t) - \kappa(t) \|  = C \Phi_{\mathbf{V}_1}(t) \, .
    \]
where \( \Phi_{\mathbf{V}_1}(t) \) is the projection error onto the fixed subspace \( \mathbf{V}_1 \), as introduced in Lemma~\ref{lemma:projErr}. Next, we observe that
\[
\partial_t \kappa 
= \mathcal{P}_{\bf V_1}\big[ \partial_t f \big] 
= \mathcal{P}_{\bf V_1} \big[ \mathcal L(f) \big] 
=  \mathcal{P}_{\bf V_1} \big[ \mathcal L(\kappa) \big] + \mathcal{P}_{\bf V_1} \big[ d \big] \, .
\]
As a result, we have
\begin{eqnarray*}
	\kappa(t_1) &=& \kappa(t_0) + \int_{t_0}^{t_1} \mathcal{P}_{\bf V_1} \big[ \mathcal L(\kappa(\tau)) \big] d\tau + \int_{t_0}^{t_1} \mathcal{P}_{\bf V_1} \big[ d(\tau) \big] d\tau  \\
		&=& y^{\rm GAP}(t_1) + \int_{t_0}^{t_1} \mathcal{P}_{\bf V_1} \big[ d(\tau) \big] d\tau \, .
\end{eqnarray*}
Here we used that \( y^{\rm GAP}(t_1) \) is defined as the solution of the projected equation without the defect term, i.e., the solution to \( \partial_t \kappa = \mathcal{P}_{\mathbf V_1}[ \mathcal{L}(\kappa) ] \) with initial condition \( \kappa(t_0) = \mathcal{P}_{\mathbf V_1}[f(t_0)] \). 
Since \( \kappa(t_1) = \mathcal P_{\mathbf V_1}[f(t_1)] \), it follows from~\eqref{eq:projLipError} that
\[
	\| \mathcal P_{\bf V_1}[f(t_1)] - y^{\rm GAP}(t_1) \|  \leq \int_{t_0}^{t_1} \| d(\tau) \| d\tau \leq C \int_{t_0}^{t_1} \Phi_{\mathbf{V}_1}(\tau) d\tau \, \lesssim  \Phi_{\mathbf{V}_1}(t_1) \Delta t + \Delta t^2 \, .
\]
The conclusion follows by observing that \( \Phi_{\mathbf{V}_1}(t_1) = \| f(t_1) - \mathcal{P}_{\mathbf{V}_1}[f(t_1)] \| \), which is of order \( \mathcal{O}(\Delta t^2) \) due to the sub-optimality of the PSI approximation, as shown in~\eqref{eq:subOptPSI}.
\end{proof}

Using a standard error propagation argument for ODEs~\cite{hairer1993solving} we obtain the following bound on the global error of the GAP algorithm.

\begin{corollary}[GAP Global Error]
    The error of the novel GAP algorithm at the final time $T = N\Delta t$, with $N \in \mathbb{N}$, satisfies
    \[
        \| y^{\rm{GAP}}(T) - f(T) \| \lesssim \Delta t + \eta + \delta_0 \, .
    \]
\end{corollary}

\section{Asymptotic-Preserving Property for Radiative Transfer}

In the following, we focus on the 1x1v scaled Radiative Transfer Equation (RTE), a fundamental model in kinetic theory that describes the transport of radiation in a scattering medium. Let \( f^\varepsilon = f^\varepsilon(t, x, \mu) \) denote its solution, governed by the equation:
\begin{equation} \label{eq:RTE}
    \partial_t f^\varepsilon + \frac{1}{\varepsilon}\mu \partial_x f^\varepsilon = \frac{1}{\varepsilon^2}(\rho^\varepsilon -f^\varepsilon) \quad \text{where} \quad \rho^\varepsilon = \frac{1}{2
    } \langle f^\varepsilon \rangle_\mu \, ,
\end{equation}
where \( x \in \Omega_x \subseteq \mathbb{R} \) and \( \mu \in [-1,1] \). Periodic boundary conditions are assumed unless stated otherwise. The parameter \( \varepsilon > 0 \) represents the Knudsen number, which characterizes the ratio between the mean free path of particles and a characteristic macroscopic length scale. 
This PDE can be embedded within the framework introduced above, where the product measure is defined as \( d\omega = dx \cdot d\mu \) on the domain \( \Omega := \Omega_x \times [-1,1] \). Within this setting, the operator \( \mathcal{L} \) is given by:
\begin{equation}
    \mathcal L(f^\varepsilon) = -\frac{1}{\varepsilon}\mu \partial_x f^\varepsilon + \frac{1}{\varepsilon^2} (\rho^\varepsilon - f^\varepsilon). 
\end{equation}
The parameter \( \varepsilon \) characterizes the scale of the system: the regime is \emph{diffusive} or \emph{collision dominated} when large scaling factors appear ($\varepsilon <\!<1$), and \emph{kinetic} otherwise $\varepsilon \approx \mathcal{O}(1)$. Consistent with the analysis in~\cite{larsen1974asymptotic, jin1999efficient, bouchut2000kinetic}, in the limit as \( \varepsilon \to 0 \), the radiative transfer equation is known to converge to the solution of
\begin{equation} \label{eq:RTE-AP}
    \partial_t \rho = \frac{1}{3} \partial_{xx} \rho \, \quad \text{with} \quad \rho = \rho(t,x) \, ,
\end{equation}
which we summarize in the following result.
\begin{theorem} [AP of RTE]
    Let $f^\varepsilon$ be the solution of \eqref{eq:RTE}. Then, 
    \[
        \lim_{\varepsilon \xrightarrow{} 0}  f^\varepsilon(t,x,\mu) =  \rho(t,x) \, ,
    \]
    where $\rho = \rho(t,x)$ denotes the solution of \eqref{eq:RTE-AP}.
\end{theorem}
We observe that in the asymptotic regime $\varepsilon \to 0$ the limiting solution $\rho(t,x)$ becomes independent of the angular variable and can thus be represented as a rank-one function in the angular direction. Since $\mathcal{L}(f^\varepsilon) = \partial_t f^\varepsilon$, in this limit the dynamics also converge to $\partial_t \rho$ which due to angular independence is itself rank-one. Consequently, the evolution of $f^\varepsilon$ approaches a low-rank manifold making negligible the component of $\mathcal{L}(f^\varepsilon)$ orthogonal to the tangent space. Hence, DLRA-assumption~(c)---which requires that the projection of $\mathcal{L}(f^\varepsilon)$ onto the orthogonal complement of the tangent space remains small---is fully consistent with the structure of the solution in this limit. Next, we demonstrate that the proposed GAP scheme satisfies the Asymptotic Preserving (AP) property. Specifically, we show that as \( \varepsilon \to 0 \), the low-rank approximation obtained from the GAP scheme correctly recovers the limiting diffusive equation \eqref{eq:RTE-AP}. This guarantees a consistent transition from the kinetic to the diffusive regime while preserving computational efficiency.
\begin{theorem}[AP for GAP]
    Let \( \kappa^\varepsilon = \kappa^\varepsilon(t,x,\mu) \in \mathcal{M}_r \) be the rank-\(r\) approximation of \eqref{eq:RTE} obtained using the novel GAP scheme, with \( r \geq 2 \). Then,
    
    \[
     	\lim_{\varepsilon \xrightarrow{} 0}   \kappa^\varepsilon(t,x,\mu) = \rho(t,x) \, ,
    \]
    where $\rho = \rho(t,x)$ denotes the solution of \eqref{eq:RTE-AP}.

    \begin{proof}
            Because $\kappa^\varepsilon \in \mathcal M_r$, we recall that it admits the following decomposition
            \[
                \kappa^\varepsilon(t,x,\mu) = \sum_{i,j=1}^r X^\varepsilon_i(t,x) S^\varepsilon_{ij}(t) V^\varepsilon_j(t, \mu) \, .
             \]
            We begin by showing that such a decomposition admits an orthonormal basis \( \mathbf V \) 
            of the form \eqref{vectorV} such that
            \[
                V_1^\varepsilon(t,\mu) \propto 1, \quad V_2^\varepsilon(t,\mu) \propto \mu \, .
            \]
            Omitting \( \varepsilon \) in the notation, the solution of the \( \ell \)-step is given by the evolution equation 
            \[
                \partial_t \ell = \mathcal P_{\mathbf X}[\mathcal L(\ell)] = \mathbf X \langle \mathbf X^\top \mathcal L(\ell) \rangle_x\, ,
            \]
            where $\mathbf X$ is of the form \eqref{vectorX}. 
			Recalling that the solution of the \( \ell \)-step lies in the span of \( \mathbf X \), we obtain
            \[
                \ell = \mathcal P_{\mathbf X}[\ell] = \mathbf X  \mathbf L^\top 
                \quad \text{where} \quad
                \mathbf L := \langle \mathbf X, \ell \rangle_x \in \mathbb{R}^{1 \times r} \, .
            \]
			Since the basis \( \mathbf{X} \) is fixed during this substep, we have \( \partial_t \mathbf{X} = 0 \), and therefore no terms involving \( \partial_t \mathbf{X} \) appear in the subsequent derivation. Thus, the evolution equation for the \( \mathbf L \)-factor is directly obtained as 
            \begin{equation} \label{eq:contL}
                \partial_t \mathbf L + \frac{1}{\varepsilon} \mu \mathbf L \langle \partial_x \mathbf X^\top \mathbf X \rangle_x = \frac{1}{2\varepsilon^2} \Big( \langle \mathbf L \rangle_\mu - 2\mathbf L\Big) \, .
            \end{equation}
			Rearranging the terms, we obtain
            \[
                \mathbf L = \frac{1}{2} \langle \mathbf L \rangle_\mu - \varepsilon \mu \mathbf L\langle \partial_x \mathbf X^\top \mathbf X \rangle_x + \mathcal O(\varepsilon^2)\, .
            \]
            Thus, up to terms of order \( \mathcal{O}(\varepsilon^2) \), the range of \( \mathbf L \) is of rank 2 and given by
            \[
                \mathbf L = [1 \  \ \mu ]   \begin{bmatrix}
                                        \frac{1}{2}\langle \mathbf L \rangle_\mu + \mathcal O(\varepsilon^2) \\ 
                                        -\varepsilon \mathbf L \langle \partial_x \mathbf X^\top \mathbf X \rangle_x \\
                                    \end{bmatrix}
                                    \, .
            \]
            Since the functions \( 1 \) and \( \mu \) are orthogonal to each other, they form a linearly independent basis. Thus, we obtain the desired result.

            \bigskip
            Next, we use the above result to show that the solution of the \( \kappa \)-step converges to the diffusive limit, i.e.
            \[
                 \lim_{\varepsilon \xrightarrow{} 0}   \kappa^\varepsilon(t,x,\mu) =  \rho(t,x) \, .
            \]
	        To simplify the notation, we omit  the subindex \( \varepsilon \) in the following. Consequently, the equation for the \( \kappa \)-step is given by:
            \[
                \partial_t \kappa
                    = \mathcal P_{\mathbf V}[\mathcal L(\kappa)] = \langle \mathcal L(k) \mathbf V \rangle_\mu \mathbf V^\top \, .      
            \]
            Recalling that the solution of the $\kappa$-step lies in the span of $\mathbf V$, we obtain
            \[
                \kappa = \mathcal P_{\mathbf V}[\kappa] = \mathbf K \mathbf V^\top
                \quad \text{where} \quad
                \mathbf K := \langle \kappa, \mathbf V \rangle_\mu \in \mathbb{R}^{1 \times r} \, .
            \]
            Since \( \mathbf{V} \) is fixed during this substep, we have \( \partial_t \mathbf{V} = 0 \), and no additional terms involving \( \partial_t \mathbf{V} \) appear in the subsequent derivation. We therefore obtain the following evolution equation for \( \mathbf{K} \):
            \begin{equation} \label{eq:contK}
                \partial_t \mathbf K +\frac{1}{\varepsilon} \partial_x \mathbf K \langle \mu \mathbf V^\top \mathbf V \rangle_\mu = \frac{1}{2\varepsilon^2} \mathbf K \Big(  \langle \mathbf V \rangle_\mu^\top \langle \mathbf V \rangle_\mu - 2\mathbf I \Big) \, .
            \end{equation}
			We begin by considering the first component of the vector \( \mathbf K \in \mathbb{R}^{1\times r}\) and obtain
            \begin{equation}
            \label{evol_K1}
                \partial_t K_1 + \frac{1}{\varepsilon} \sum_{j\geq 1} \langle \mu V_1 V_j \rangle_\mu \partial_x K_j = \frac{1}{2\varepsilon^2}\Big[ \sum_{j\geq 1}  \langle V_1\rangle_\mu \langle V_j\rangle_\mu K_j - 2K_1 \Big] \, .
            \end{equation}
			Since from the $\ell$-step we obtained \( V_1 \propto 1 \) and \(  V_2 \propto \mu \), we observe that
            \begin{align*}    
                \sum_{j\geq 1} \langle \mu V_1 V_j \rangle_\mu \partial_x K_j 
                    &= \underbrace{\langle \mu V_1 V_1 \rangle_\mu}_{\propto \langle \mu \rangle = 0} \partial_x K_1 
                    + \langle \mu V_1 V_2 \rangle_\mu \partial_x K_2
                    + \sum_{j \geq 3} \underbrace{\langle \mu V_1 V_j \rangle_\mu}_{\propto \langle V_2 V_j \rangle = \delta_{2j}} \partial_x K_j \\
                    &= \langle \mu V_1 V_2 \rangle_\mu \partial_x K_2 \, .
            \end{align*}
           On the right-hand side of \eqref{evol_K1} 
           we obtain that
            \[
                \sum_{j \geq 1} \langle V_1\rangle_\mu \langle V_j\rangle_\mu K_j - 2K_1  
                    = \underbrace{\langle V_1 \rangle_\mu^2}_{=2}K_1 + \cancel{\sum_{j \geq 2} \langle V_1 \rangle_\mu \langle V_j \rangle_\mu K_j}  - 2K_1
                    = 0 \, ,
            \]
            where the term cancels due to the orthogonality of the elements of \( \mathbf V \), i.e.,
            $$\langle V_j \rangle_\mu = \langle 1 \cdot V_j \rangle_\mu \propto \langle V_1 \cdot V_j \rangle_\mu =  \delta_{1j} \, .$$
            To summarize, we conclude that the first component \( K_1 \) satisfies
            \begin{equation}
            \label{evol_K1_2}
                \partial_t K_1 +\frac{1}{\varepsilon} \langle \mu V_1 V_2 \rangle_\mu \partial_x K_2 = 0 \, .
            \end{equation}
            For the remaining components (\( i = 2, \dots, r \)) of \eqref{eq:contK}, we similarly obtain 
            \[
                \partial_t K_i +\frac{1}{\varepsilon} \langle \mu V_i V_1 \rangle_\mu \cdot \partial_x K_1 +\frac{1}{\varepsilon} \sum_{j \geq 2} \langle \mu V_i V_j \rangle_\mu \cdot \partial_x K_j = - \frac{1}{ \varepsilon^2}  K_i \, . 
            \]
			The latter can be restated as
            \[
                K_i = -\varepsilon \langle \mu V_1 V_i \rangle_\mu \partial_x K_1   -\varepsilon \sum_{j \geq 2} \langle \mu V_i V_j \rangle_\mu \cdot \partial_x K_j + \mathcal{O}(\varepsilon^2) \, .
            \]
           For \( i = 2 \), the latter translates to
            \begin{align*}    
                K_2 &=   -\varepsilon \langle \mu V_1 V_2 \rangle_\mu \partial_x K_1
                        -\varepsilon \underbrace{\langle \mu V_2 V_2 }_{\propto \langle \mu^3 \rangle_\mu = 0} \rangle_\mu \partial_x K_2
                        -\varepsilon \sum_{j \geq 3} \langle \mu V_2 V_j \rangle_\mu \cdot \underbrace{\partial_x K_j}_{=\mathcal O(\varepsilon)} + \mathcal{O}(\varepsilon^2) 
                    \\
                    &=  -\varepsilon \langle \mu V_1 V_2 \rangle_\mu \partial_x K_1 + \mathcal{O}(\varepsilon^2)\, .
            \end{align*}
            Thus, the conclusion follows by substituting the above result into \eqref{evol_K1_2}, i.e. 
            \[
                \partial_t K_1 - \langle \mu V_1 V_2 \rangle_\mu^2 \cdot \partial_{xx} K_1 = \mathcal O(\varepsilon^2) \, .
            \]
            To conclude, we observe that
            \[
                \langle \mu V_1 V_2 \rangle_\mu^2 = \frac{1}{\|V_1\|^2} \frac{1}{\|V_2\|^2} \Big(\langle \mu^2 \rangle\Big)^2 =  \frac{1}{2} \cdot \frac{3}{2} \cdot \Big( \frac{2}{3} \Big)^2  = \frac{1}{3} \, .
            \]
           	Thus, in the limit as \( \varepsilon \to 0 \), the AP limit \eqref{eq:RTE-AP} is recovered.
            \qed
    \end{proof} 
\end{theorem}

\begin{remark}
		We recall that the decomposition of \(\kappa^\varepsilon\) into a low-rank representation is not unique, as the tangent space and corresponding basis \(\mathbf{V}\) can be chosen up to rotations within the subspace spanned by the angular modes. However, in the asymptotic regime \(\varepsilon \to 0\), the solution approaches a rank-one profile in angle, and the angular dependence becomes increasingly aligned with the first few Legendre polynomials. In particular, the dominant modes are proportional to \(1\) and \(\mu\), corresponding to the zeroth and first angular moments (density and flux). In the proof only the span of \(\mathbf{V}\) is important. However, to more clearly see what happens it is instructive to choose the following orthonormal basis
		\[
		V_1^\varepsilon(t,\mu) \propto 1, \quad V_2^\varepsilon(t,\mu) \propto \mu \, .
		\]
        In the diffusion limit we then get explicitly $\lim_{\varepsilon \to 0} K_1^{\varepsilon}(t,x) = \rho(t,x)$ and all other components of \(\mathbf{K}\) vanish. If desired, this behavior can be enforced in the numerical implementation by explicitly setting \(V_1^{\varepsilon}(t,\mu) \propto 1\) and \( V_2^{\varepsilon}(t,\mu) \propto \mu\) and normalizing the remaining basis functions against them.
        
\end{remark}

\section{Discretization and numerical implementation}

While the preceding sections have focused on the continuous formulation of the scaled 1x1v Radiative Transfer Equation (RTE), the notation adopted so far naturally extends to the discrete setting. In particular, spatial and angular quantities such as \(\partial_x\) and \( \langle \cdot \rangle_\mu\) can be interpreted as their discrete counterparts without requiring changes to the underlying formulation. It suffices to define spatial and angular discrete domains according to the chosen discretization. This consistency allows us to defer discretization choices until the implementation stage, which is advantageous when analyzing the asymptotic behavior of the method. However, for clarity and to aid reproducibility, we now provide a concrete description of the discretization strategies employed in our numerical experiments. This includes the transition to matrix-based notation, which is more convenient for practical implementation. 
Importantly, the asymptotic analysis applies in the discrete
setting as well, with the only modification that the scheme converges to the
discrete limit of the \(\rho\)-equation \eqref{eq:RTE-AP}, i.e. equation \eqref{eq:RTE-AP} needs to be interpreted with $\partial_{xx}$ being given by the chosen spatial discretization (finite difference, finite elements, ...).
To begin the discretization process, we first recall the scaled 1x1v Radiative Transfer Equation (RTE) in its continuous form. The equation governs the evolution of the intensity \( f^\varepsilon(t, x, \mu) \) and is given by
\[
\partial_t f^\varepsilon + \frac{1}{\varepsilon} \mu \partial_x f^\varepsilon = \frac{1}{\varepsilon^2}(\rho^\varepsilon - f^\varepsilon), \quad \text{where} \quad \rho^\varepsilon = \frac{1}{2} \langle f^\varepsilon \rangle_\mu.
\]
Here, \( \varepsilon \) represents the scaling parameter that characterizes the system, transitioning the equation from a kinetic regime to a diffusive regime as \( \varepsilon \to 0 \). The domain of the equation is \( x \in \Omega_x \subseteq \mathbb{R} \) and \( \mu \in [-1, 1] \), with periodic boundary conditions assumed unless stated otherwise.

To discretize the scaled RTE, we employ \emph{finite differences} for the spatial discretization and \emph{Gauss--Legendre collocation points} for the angular discretization. In the spatial domain, the function \( f^\varepsilon(t, x, \mu) \) is approximated using a uniform grid in \( x \), denoted by \( x_i \in \Omega_x \). For the angular discretization, the angular variable \( \mu \) is discretized using Gauss--Legendre quadrature points, which provide an efficient approximation of integrals over \( \mu \in [-1, 1] \). This approach allows us to obtain a set of discrete values for \( f^\varepsilon(t, x_i, \mu_j) \), where \( \mu_j \) are the collocation points corresponding to the chosen order of quadrature. We select \( N_x \) points for the spatial dimension and \( N_\mu \) points for the angular dimension.

To maintain a matrix-vector notation as proposed in the provided framework, we collect all the discrete values of the function \( f^\varepsilon(t, x, \mu) \) in a matrix. Let \[ \mathbf{F} = \mathbf{F}(t) \in \mathbb{R}^{N_x \times N_\mu} \] represent the matrix of all discrete values, where each entry corresponds to the discretized value of \( f^\varepsilon(t, x_i, \mu_j) \). The goal is to derive a matrix differential equation that represents the semi-discrete form of the scaled 1x1v RTE. This approach allows us to treat the problem in a more compact and efficient form, where the evolution of \( \mathbf{F}(t) \) over time can be described using matrix operations, facilitating the computation of the solution. For the sake of notation, we omit $\varepsilon$ in the matrix formulation.

Now, if we introduce the column vector \( \mathbf{w}_\mu \in \mathbb{R}^{N_\mu \times 1} \) containing the Gauss--Legendre weights required for the quadrature in \( \mu \), the integral quantities in angle can be expressed as
\[
	\rho^\varepsilon = \frac{1}{2} \langle f^\varepsilon \rangle_\mu  \approx \frac{1}{2} \sum_j \omega_j \cdot f^\varepsilon(t, x_i, \mu_j) = \frac{1}{2} \mathbf{F} \mathbf{w}_\mu \mathbf{1}^\top \, ,   
\]
where \( \mathbf{1} \in \mathbb{R}^{N_\mu \times 1 } \) is a column vector of ones, used to express that the \( \rho^{ \varepsilon} \) function corresponds to the $\mu-$averaged value of the function \( f^{ \varepsilon} \) at each point of the space grid. Thus, the discretized scaled 1x1v RTE reduces to the following matrix differential equation
\begin{equation} \label{eq:discRTE}
	\dot{\mathbf{F}} + \frac{1}{\varepsilon} \mathbf{D}_x \mathbf{F} \, \operatorname{diag}(\boldsymbol{\mu}) = \frac{1}{\varepsilon^2} \left(  \mathbf{F} \mathbf{W_\mu}  - \mathbf{F} \right) \, ,
\end{equation}
where \( \mathbf{W_\mu} := \frac{1}{2} \mathbf{w}_\mu \mathbf{1}^\top \in \mathbb{R}^{N_\mu \times N_\mu} \) represents the matrix that computes the average over \( \mu \), and \( \mathbf{D}_x \in \mathbb{R}^{N_x \times N_x} \) represents a suitable discretization of the first spatial derivative, incorporating the periodic boundary conditions. The matrix \( \operatorname{diag}(\boldsymbol{\mu}) \) is a diagonal matrix of size \( N_\mu \times N_\mu \), where each diagonal element corresponds to an element from the vector \( \boldsymbol{\mu} \), representing the discrete angular points in the Gauss--Legendre quadrature.

\subsection{Discretized GAP substeps}
In this section, we present the discretization of the intermediate substeps of the GAP scheme. Building upon the discrete formulation \eqref{eq:discRTE}, we provide the discretization of each substep of the algorithm to facilitate practical computations. The discretization in space and angle is implemented using the methods described in the previous section. Specifically, we utilize finite differences in space and Gauss--Legendre quadrature for the angular part.


We recall that, at the continuous level, the evolution equations for the  \(\ell\)-step of \eqref{eq:contL} and the \(\kappa\)-step of \eqref{eq:contK} in the proposed framework are given by
\begin{align}
\label{sec5:LK}
	\left\{
	\begin{array}{l}
		\partial_t \mathbf L + \frac{1}{\varepsilon} \mu \mathbf L \langle \partial_x \mathbf X^\top \mathbf X \rangle_x = \frac{1}{2\varepsilon^2} \Big( \langle \mathbf L \rangle_\mu - 2\mathbf L\Big) \, , \\[1.2ex]
		\partial_t \mathbf K +\frac{1}{\varepsilon} \partial_x \mathbf K \langle \mu \mathbf V^\top \mathbf V \rangle_\mu = \frac{1}{2\varepsilon^2} \mathbf K \Big(  \langle \mathbf V \rangle_\mu^\top \langle \mathbf V \rangle_\mu - 2\mathbf I \Big) \, .
	\end{array}
	\right.
\end{align}
where the bases \(\mathbf{X}\) and \(\mathbf{V}\) are orthonormal with respect to the scalar products defined on the spatial and angular domains. Before deriving the discretized counterparts of these evolution equations for the substeps, we first introduce the scalar products in discrete space and angle. 
To keep the notation as close as possible to the continuous setting, we define the following discrete inner products for matrices \( \mathbf{X}, \mathbf{Z} \in \mathbb{R}^{N_x \times r} \) and \( \mathbf{V}, \mathbf{M} \in \mathbb{R}^{N_\mu \times r} \):
\[
\langle \mathbf{X}^\top \mathbf{Z} \rangle_x := \mathbf{X}^\top \operatorname{diag}(\mathbf{\Delta x}) \mathbf{Z}, \qquad
\langle \mathbf{V}^\top \mathbf{M} \rangle_\mu := \mathbf{V}^\top \operatorname{diag}(\mathbf{w_\mu}) \mathbf{M} \, .
\]
Here, \( \Delta x \) represents the spacing between points in the spatial grid, and \( \mathbf{w_\mu} \) denotes the Gauss--Legendre weights for the angular discretization. Additionally, \( \operatorname{diag}(\mathbf{\Delta x}) \) represents a diagonal matrix of size \( N_x \times N_x \), where each diagonal element corresponds to the discretization step \( \Delta x \) in space. Similarly, \( \mathbf{w_\mu} \) is a column vector of size \( N_\mu \times 1 \), and \( \operatorname{diag}(\mathbf{w_\mu}) \) represents a diagonal matrix of size \( N_\mu \times N_\mu \), where each diagonal element corresponds to the associated weight in the Gauss--Legendre quadrature. These diagonal matrices induce scalar products that can be used to define a modified Gram-Schmidt procedure for orthonormalizing the vectors in space and angle, respectively, as prescribed at the end of each sub-step in the GAP-scheme.

We begin by deriving the discretized evolution equation for the L-step of \eqref{eq:contL}. We fix the range of the numerical approximation, i.e.
\[
	 \mathbf{F}(t) \approx \mathbf{X}_0 \mathbf{L}(t)^\top \in  \mathcal M_r, 
\]
where the range is orthonormal with respect to the discretized scalar product, i.e.,
\[
	\langle \mathbf{X}_0^\top \mathbf{X}_0 \rangle_x = \mathbf{I}_r  \, .
\]
Inserting the ansatz  $\mathbf{F}(t) \approx \mathbf{X}_0 \mathbf{L}(t)^\top$  into the full discretized equation \eqref{eq:discRTE}, we obtain:
\[
	\mathbf{X}_0 \dot{\mathbf{L}}^\top + \frac{1}{\varepsilon} \mathbf{D}_x \mathbf{X}_0 \mathbf{L}^\top \operatorname{diag}(\boldsymbol{\mu}) = 
	\frac{1}{\varepsilon^2}\big( \mathbf{X}_0 \mathbf{L}^\top \mathbf{W_\mu} - \mathbf{X}_0 \mathbf{L}^\top \big) \, .
\]
To isolate the evolution equation for the \( \mathbf L \)-factor, we multiply both sides from the left by \( \mathbf{X}_0^\top \operatorname{diag}(\mathbf{\Delta x}) \), obtaining
\begin{equation} \label{eq:discL}
	\dot{\mathbf{L}}^\top + \frac{1}{\varepsilon}  \mathbf{A}_x \mathbf{L}^\top \operatorname{diag}(\boldsymbol{\mu}) = 
\frac{1}{\varepsilon^2}\Big(  \mathbf{L}^\top \mathbf{W_\mu} - \mathbf{L}^\top \Big) \, ,
\end{equation}
where the matrix $\mathbf{A}_x \in \mathbb{R}^{r \times r}$ is given by
\[
	\mathbf{A}_x := \mathbf{X}_0^\top \operatorname{diag}(\mathbf{\Delta x}) \mathbf{D}_x \ \mathbf{X}_0   \, .
\]
The matrix \(\mathbf{A}_x\) represents the discrete counterpart of the quantity \(\langle \partial_x \mathbf{X}^\top \mathbf{X} \rangle_x\) appearing in the continuous formulation \eqref{sec5:LK}. 
In a similar manner, by fixing the co-range, i.e., \( \mathbf{F}(t) \approx \mathbf{K}(t) \mathbf{V}_1^\top \), we obtain the discrete equivalent of \eqref{eq:contK} for the \( \mathbf{K} \)-factor by inserting this ansatz into \eqref{eq:discRTE} and multiplying from the right by \( \operatorname{diag}(\mathbf{w_\mu}) \mathbf{V}_1 \):
\begin{equation} \label{eq:discK}
	\dot{\mathbf{K}} + \frac{1}{\varepsilon} \mathbf{D}_x \, \mathbf{K} \, \mathbf{B}_\mu = \frac{1}{ \varepsilon^2} \Big( \mathbf{K}\mathbf{C}_\mu - \mathbf{K}\Big)	 \, ,
\end{equation}
where the matrices $\mathbf{B}_\mu, \mathbf{C}_\mu \in \mathbb{R}^{r \times r}$ are defined by 
\begin{align*}
	\mathbf{B}_\mu &:= \mathbf{V}_1^\top \operatorname{diag}(\boldsymbol{\mu}) \operatorname{diag}(\mathbf{w_\mu}) \mathbf{V}_1  , \\
	\mathbf{C}_\mu &:= \mathbf{V}_1^\top \mathbf{W}_\mu \operatorname{diag}(\mathbf{w_\mu}) \mathbf{V}_1 \, .
\end{align*}
The matrices \(\mathbf{B}_\mu\) and \(\mathbf{C}_\mu\) represent the discrete counterparts of the continuous quantities \(\langle \mu \mathbf{V}^\top \mathbf{V} \rangle_\mu\) and \(\langle \mathbf{V} \rangle_\mu^\top \langle \mathbf{V} \rangle_\mu\) occurring in \eqref{sec5:LK}, respectively. Finally, we reiterate that the orthonormalization of the \( \mathbf K \)-factor and the \( \mathbf L \)-factor must be performed with respect to the scalar products induced by their continuous counterparts. From a discretization perspective, this requires performing the orthonormalization, for example, via a modified Gram-Schmidt procedure, using the scalar products induced by the diagonal matrices \( \operatorname{diag}(\mathbf{\Delta x}) \) in space for the \( \mathbf K \)-factor and \( \operatorname{diag}(\mathbf{w_\mu}) \) in angle for the \( \mathbf L \)-factor, respectively.



\subsection{Time discretization}
%
The time-integration of the scaled 1x1v Radiative Transfer Equation (RTE) presents significant challenges due to the presence of stiff terms, which scale with magnitude \( \frac{1}{\varepsilon^2} \) as \( \varepsilon \) tends to zero. As a result, explicit time-integration schemes may struggle, leading to stability issues and the need for very small time steps. To address these difficulties, an implicit time-integration method or a numerically exact solver is required for the substeps at least when an asymptotic preserving scheme is desired. 

To solve the system, it is sufficient to vectorize the solution of the evolution equations for \( \mathbf{L} \) and \( \mathbf{K} \). This vectorization process leads to the construction of the linear propagation operators, which allows for efficient computation of the system’s solution. We thus obtain from \eqref{eq:discL} and \eqref{eq:discK}, the linear operators
\begin{align*}
	\mathcal{A}_L &:= 
		-\frac{1}{\varepsilon} \, \mathbf{A}_x \otimes \operatorname{diag}(\boldsymbol{\mu}) 
		+ \frac{1}{\varepsilon^2} \Big( \mathbf{I}_r \otimes \mathbf{W}_\mu^\top
		-  \mathbf{I}_r \otimes \mathbf{I}_{N_\mu} \Big) 
		\in \mathbb{R}^{ r N_\mu \times r N_\mu}
		\, , \\
	\mathcal{A}_K &:= 
		-\frac{1}{\varepsilon} \mathbf{B}_\mu^\top \otimes \mathbf{D}_x 
		+\frac{1}{\varepsilon^2} \Big(\mathbf{C}_\mu^\top \otimes \mathbf{I}_{N_x}
		- \mathbf{I}_r \otimes \mathbf{I}_{N_x} \Big) 
		\in \mathbb{R}^{r N_x \times r N_x}
	\, .	
\end{align*}
Using the notations \(\mathcal{A}_L\) and \(\mathcal{A}_K\) introduced above, and letting \({\boldsymbol\ell} \in \mathbb{R}^{r N_\mu}\) and \({\mathbf k} \in \mathbb{R}^{r N_x}\) denote the vectorizations of \(\mathbf{L} \in \mathbb{R}^{r \times N_\mu}\) and \(\mathbf{K} \in \mathbb{R}^{N_x \times r}\), respectively, equations~\eqref{eq:discL} and~\eqref{eq:discK} can be recast in vectorized form as
\[
	\dot{{\boldsymbol\ell}} = \mathcal{A}_L {\boldsymbol\ell}, 
	\qquad 
	\dot{\mathbf{k}} = \mathcal{A}_K \mathbf{k} \, .
\]

Thus, the solution to the \( \mathbf K \)-step and \( \mathbf L \)-step is reformulated in terms of the time-evolution of their vectorized forms. While constructing the provided operators is generally inconvenient, we note that this is not the case here. First, the rank \(r\) is assumed to be small, meaning that the solution to each problem is linear in terms of the discretization size. Second, the operators are sparse, making their construction feasible and compatible with implicit or numerically exact solvers. Specifically, we note that the  \(\bf K\)-step evolution time-integration problem, for instance, can be solved implicitly via
\[
	\mathbf{k}_{n+1} = \mathbf{k}_n + \Delta t \mathcal A_K \mathbf{k}_{n+1} \, ,
\]
where \( \mathbf{k}_n \) represents the vectorized solution \( \mathbf{K}(t_n) \) at time \( t_n = t_0 + n \Delta t \), \( \mathbf{k}_{n+1} \) is the solution at the next time step \( t_{n+1} \). Similarly, this can be done for the \( \mathbf L \)-step. Thus, it is sufficient to solve iteratively the problem  

\[
(\mathbf{I} - \Delta t \mathcal{A}_K) \mathbf{k}_{n+1} = \mathbf{k}_n \, .
\]
Alternatively, the problem can be solved by directly computing the exponential of the operator \( \Delta t \mathcal{A}_K \) applied to the vector \( \mathbf{k}_n \), i.e.

\[
\mathbf{k}_{n+1} = \exp(\Delta t \mathcal{A}_K) \mathbf{k}_n \, .
\]
A similar exponential formulation applies to the \( \mathbf{L} \)-step, where the corresponding matrix exponential is applied to the vectorized form of \( \mathbf{L} \), as specified in Algorithm~\ref{algoGAP_final}:
$$
 {\boldsymbol\ell}_{n+1} = \exp(\Delta t \mathcal{A}_L)  {\boldsymbol\ell}_{n}.
$$


We adopt this time-integration strategy because it preserves the asymptotic-preserving (AP) property, as the numerical flow is evaluated exactly and thus remains consistent with the asymptotic analysis in the limit \(\varepsilon \to 0\). Due to the small dimensionality and sparsity of the linear operator, this exponential action can be efficiently and accurately computed using methods such as Krylov subspace techniques or Leja interpolation. This is the trivial case of what are known as exponential integrators and we refer to the review article \cite{Hochbruck2010} for more details.

\begin{remark}
	A key feature of the GAP scheme, as presented here, is that in the limit the $\ell$-step gives the desired basis, while the $k$-step gives the dynamics of $\rho$. Since both steps are linear it is natural to directly compute the action of the matrix exponential (i.e.~use an exponential integrator). Then there is no CFL type time step restriction of the limit equation and well prepared initial data (as described in \cite{Ding2021}) are not required. In the diffusion limit, the $\ell$-step only prepares the basis. Thus, the time integration error is exclusively due to the $k$-step, which is only limited by the accuracy of the linear algebra routine used to compute the matrix exponential. One can also replace the computation of the exponential by, for example, an L-stable implicit scheme. In this case we still are free of a CFL condition and do not require well prepared initial data. In the limit the order of the implicit method determines how accurately the equation for $\rho$ is solved. In practical implementations, the use of (preconditioned) Krylov subspace methods to evaluate the action of the exponential or the linear solve can significantly enhance both efficiency and robustness. Given the scope of the present work, a more systematic investigation of such technique and its impact on performance is not pursued here.
\end{remark}
 
\subsection{The fully discretized GAP scheme}
In the following, we detail the discretized GAP scheme in the form of the algorithm presented. This version of the scheme is tailored for practical implementation, where time-integration is performed using the exponential method. Using an implicit time-integration method is also possible and requires only minor modifications. The algorithm proceeds by iteratively evolving the low-rank approximations of the \( \mathbf K \)- and \( \mathbf L \)-factors, leveraging the evolution equations that govern their dynamics.

The Modified Gram-Schmidt (MGS) function, employed at multiple stages of the algorithm, takes two arguments: a matrix and the corresponding scalar product. The MGS procedure here is used for orthonormalizing the vectors with respect to a general scalar product. Specifically, in the discretized setting, the scalar product is induced by the diagonal matrices \( \operatorname{diag}(\mathbf{w}_\mu) \) for the angular direction and \( \operatorname{diag}(\mathbf{\Delta x}) \) for the spatial direction. These matrices enable the MGS procedure to account for the geometry of the discretized space and the angular discretization, ensuring that the orthonormalization respects the underlying structure. 

\begin{algorithm}[H]
	\caption{The discretized GAP Scheme}
	\begin{algorithmic}[1]
		
		\STATE \textbf{Input:} \( \mathbf{F}(t_0) = \mathbf{X}_0 \mathbf{S}_0 \mathbf{V}_0^\top \in \mathcal{M}_r, \ N_T \in \mathbb{N} \, . \)
		\vspace{0.5em}
		
		\FOR{$i = 1, \dots N_T$}
		\vspace{0.5em}
		
		\hrule
		
		\vspace{0.2em}
		\STATE \texttt{// \(\mathbf{L}\)-step: Update angular part}
		\vspace{0.2em}
		\hrule
		\vspace{0.5em}
		
		\STATE Set \( {\boldsymbol\ell}_{i-1} = \text{vec}(\mathbf{V}_{i-1} \mathbf{S}_{i-1}^\top) \, . \)
		\vspace{0.5em}

		\STATE Update the matrix \(\mathbf{A}_x\) using \(\mathbf{X}_{i-1}\).
		\vspace{0.5em}
		
		\STATE Let \( \boldsymbol\ell_i = \exp(\Delta t \mathcal{A}_L) \boldsymbol\ell_{i-1} \, . \)
		\vspace{0.5em}
		
		\STATE Matricize \( \boldsymbol\ell_i \) into a matrix \( \mathbf{L}_i \in \mathbb{R}^{N_\mu \times r}  \, . \)
		\vspace{0.5em}

		
		\STATE Let \([ \mathbf{V}_i, \, \cdot \, ] = \text{MGS}(\mathbf{L}_i, \operatorname{diag}(\mathbf{w}_\mu)) \) .
		\vspace{1em}

		\hrule
		\vspace{0.2em}
		\STATE \texttt{// \(\mathbf{K}\)-step: Update spatial part}
		\vspace{0.2em}
		\hrule
		\vspace{0.5em}
		
		\STATE Set \( \mathbf{K}_{i-1} = \mathbf{X}_{i-1} \mathbf{S}_{i-1} (\mathbf{V}_{i-1}^\top \operatorname{diag}(\mathbf{w}_\mu) \mathbf{V}_i)  \, . \)
		\vspace{0.5em}
		
		\STATE Let \( \mathbf{k}_{i-1} = \text{vec}(\mathbf{K}_{i-1})  \, . \)
		\vspace{0.5em}
		
		\STATE Update the matrices \(\mathbf{B}_\mu\) and \(\mathbf{C}_\mu\) via the new \(\mathbf{V}_i\).
		
		\vspace{0.5em}
		
		\STATE Set \( \mathbf{k}_i = \exp(\Delta t \mathcal{A}_K) \mathbf{k}_{i-1}  \, . \)
		\vspace{0.5em}
		
		\STATE Matricize \( \mathbf{k}_i \) into a matrix \( \mathbf{K}_i \in \mathbb{R}^{N_x \times r} \, . \)
		\vspace{0.5em}
		
		\STATE Let \([ \mathbf{X}_i, \mathbf{S}_i ] = \text{MGS}(\mathbf{K}_i, \operatorname{diag}(\mathbf{\Delta x})) \, . \)
		\vspace{1em}
		
		\ENDFOR
		

	\end{algorithmic}
	\label{algoGAP_final}
\end{algorithm}

The value \( N_T \) in the scheme denotes the total number of time iterations required to reach the final time \( T \) using time step \( \Delta t \). The final output of the algorithm is the low-rank matrix 
\[  \mathbf{X}_{N_T} \mathbf{S}_{N_T} \mathbf{V}_{N_T}^\top \approx \mathbf{F}(T)  \, . \]

To conclude, numerical evidence suggests that the chosen time-integration method allows for large time steps without compromising the asymptotic-preserving (AP) behavior or the accuracy of the GAP scheme.
As another key component of the algorithm, we note that when used without pivoting, the standard MGS process is generally not backward stable~\cite{higham2002accuracy}. A more robust approach is to apply MGS to the orthonormal factor obtained from a backward stable QR decomposition -- such as the Householder QR implementation provided in the LAPACK library -- and where appropriate absorb the remaining factors into the other components of the low-rank representation.

\section{Numerical results}


In this section, we assess the performance of the proposed discretized GAP scheme and verify its asymptotic-preserving (AP) property in the diffusive limit. The scheme, described in the previous section, is implemented in \textsc{Matlab} R2024b and executed on a workstation equipped with an Intel i7-11700 processor and 32\,GB of RAM. For time integration, the action of the matrix exponentials of the linear operators \( \mathcal{A}_L \) and \( \mathcal{A}_K \) is computed using \textsc{Matlab}'s built-in \texttt{expmv} function, which allows for accurate and efficient integration without introducing stiffness-related constraints. 
In the following, the spatial operators \( \partial_x \) and \( \partial_{xx} \) are discretized using centered finite differences on a uniform grid with \( N_x \) points under periodic boundary conditions and denoted by \( \mathbf{D}_x \) and \( \mathbf{D}_{xx} \), respectively.

\subsection{Asymptotic-preserving property in the diffusive limit}

For this numerical experiments, we select the following parameters for the discretization: \( N_x = 1000 \), \( N_\mu = 100 \), and the rank \( r = 5 \). The domain is \( \Omega_x = [0, 2] \), and the initial condition is chosen as \( f_0(x, \mu) = \left( (x-1)^2 + 1 \right) \left( 1 + \mu^2 \right) \). The final time is \( T = 1 \), and the time-step size is \( \Delta t = 0.1 \). To explore the asymptotic behavior, we let the scaling parameter \( \varepsilon \) vary from \( 1 \) to \( 10^{-4} \). The large value \( N_x = 1000 \) is chosen to illustrate the viability of the numerical scheme under high spatial resolution. This setting is representative of large-scale applications where fine discretization is required. The choice \( N_\mu = 100 \) provides a reasonable angular resolution for the regime considered.

\begin{figure}
	\centering
    \includegraphics[width=\textwidth]{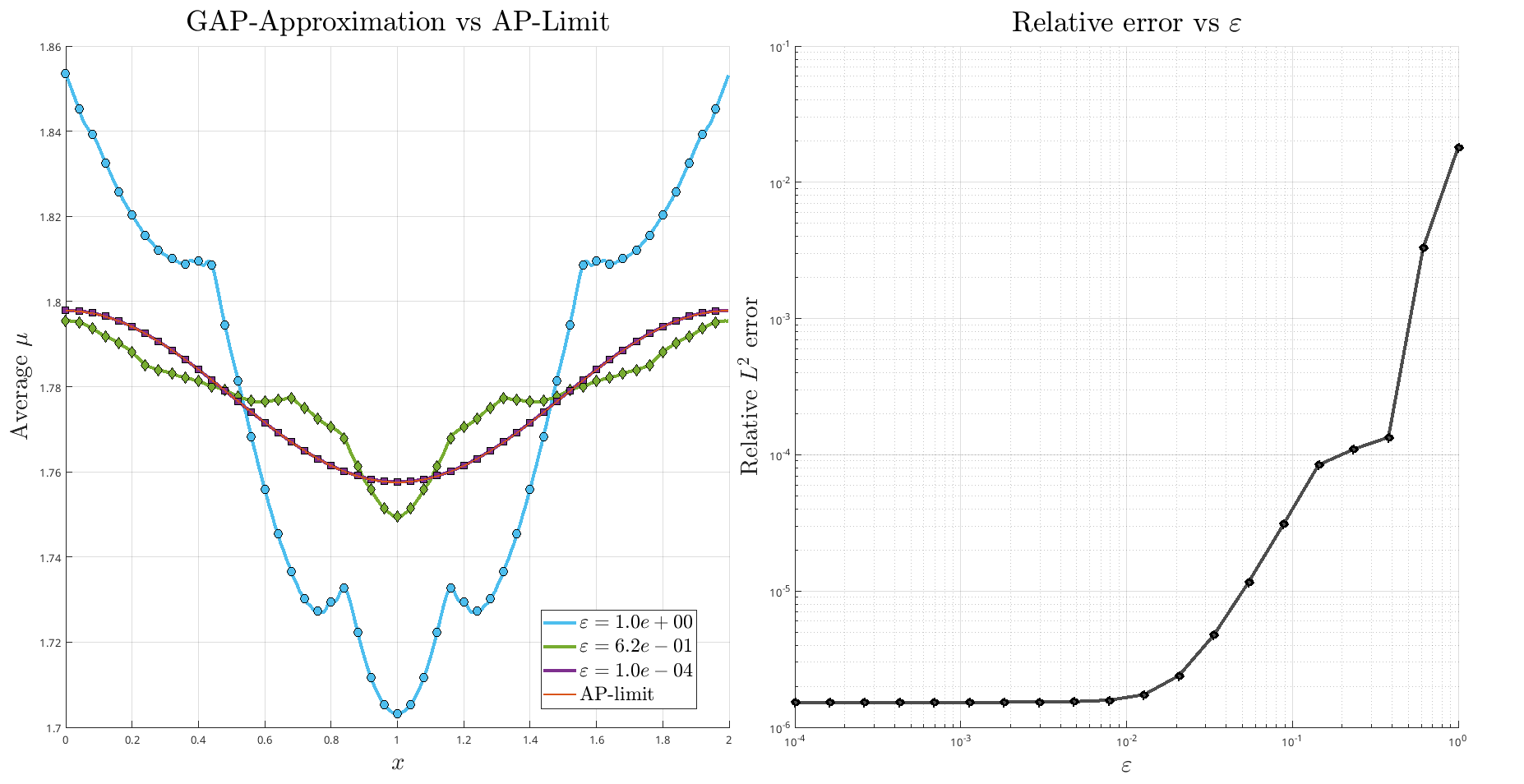}
    \caption{
    Left: GAP numerical approximation for representative values of \(\varepsilon\), including the AP-limit.
    Right: Relative \(L^2\) error between the GAP approximation and the discrete AP-limit as a function of \(\varepsilon\). The parameters used are \(N_x = 1000\), \(N_\mu = 100\), rank \(r = 5\), and time step \(\Delta t = 0.1\). The exponential of the linear operators \(\mathcal{A}_L\) and \(\mathcal{A}_K\) is computed using the \texttt{expmv} function.}
	\label{fig:GAP_limit}
\end{figure}

Figure~\ref{fig:GAP_limit} illustrates the behavior of the GAP method across different regimes. The left panel shows the discrete \(\mu\)-averaged density \( \langle y^{\rm GAP}(T) \rangle_\mu \) at final time for various \( \varepsilon \)-values compared with the discrete diffusion-limit \( \rho(T) = \exp\left( \tfrac{T}{3} \mathbf{D}_{xx} \right) \langle f_0 \rangle_\mu \). 
The right panel shows the spatial \( L^2(\Omega_x) \) relative error at final time as a function of \( \varepsilon \).

The results shown in Figure~\ref{fig:GAP_limit} confirm that the low-rank numerical approximation produced by the GAP scheme converges to the discrete diffusion limit as \(\varepsilon \to 0\), for a fixed set of numerical parameters, thereby demonstrating the asymptotic-preserving (AP) property of the method. When the action of the exponential of the linear operators \(\mathcal{A}_L\) and \(\mathcal{A}_K\) is computed using the \texttt{expmv} function, the scheme remains stable and accurate even for relatively large time steps such as \(\Delta t = 0.1\). 
%
%
The relative error saturates around \(10^{-6}\), consistent with the expected spatial accuracy of the GAP scheme. The transport operator in GAP is discretized using a second-order centered finite difference scheme on a uniform grid, which in the diffusive limit induces a discrete approximation that differs from the standard three-point stencil used for the AP-limit --- computed using \texttt{expm} in \textsc{Matlab}. This mismatch in spatial discretization explains the asymptotic error plateau.

We remark that the chosen time-step size \( \Delta t = 0.1 \) is orders of magnitude larger than what would be allowed by a standard CFL condition in the diffusive regime. The observed stability and accuracy of the GAP scheme in this setting further demonstrate its asymptotic preserving nature and highlight the advantage of using exponential integrators that avoid parabolic-type stiffness constraints.

\subsection{Time-step convergence at fixed $\varepsilon$}

	To further assess the temporal accuracy of the scheme in the kinetic regime, we compute the solution at fixed \(\varepsilon = 1\) while varying the time-step size \(\Delta t\). Here, the domain remains \( \Omega_x = [0, 2] \), and we set \(N_x = 200\), \(N_\mu = 100\), and rank \(r = 10\), using the initial condition
	\[
	f_0(x, \mu) = 1 + \sum_{k=1}^{10} 10^{-k} \sin(k\pi x)\, \mu^k.
	\]

	\begin{figure} \label{fig:pointwiseErrorConvergence}
		\includegraphics[width=\textwidth]{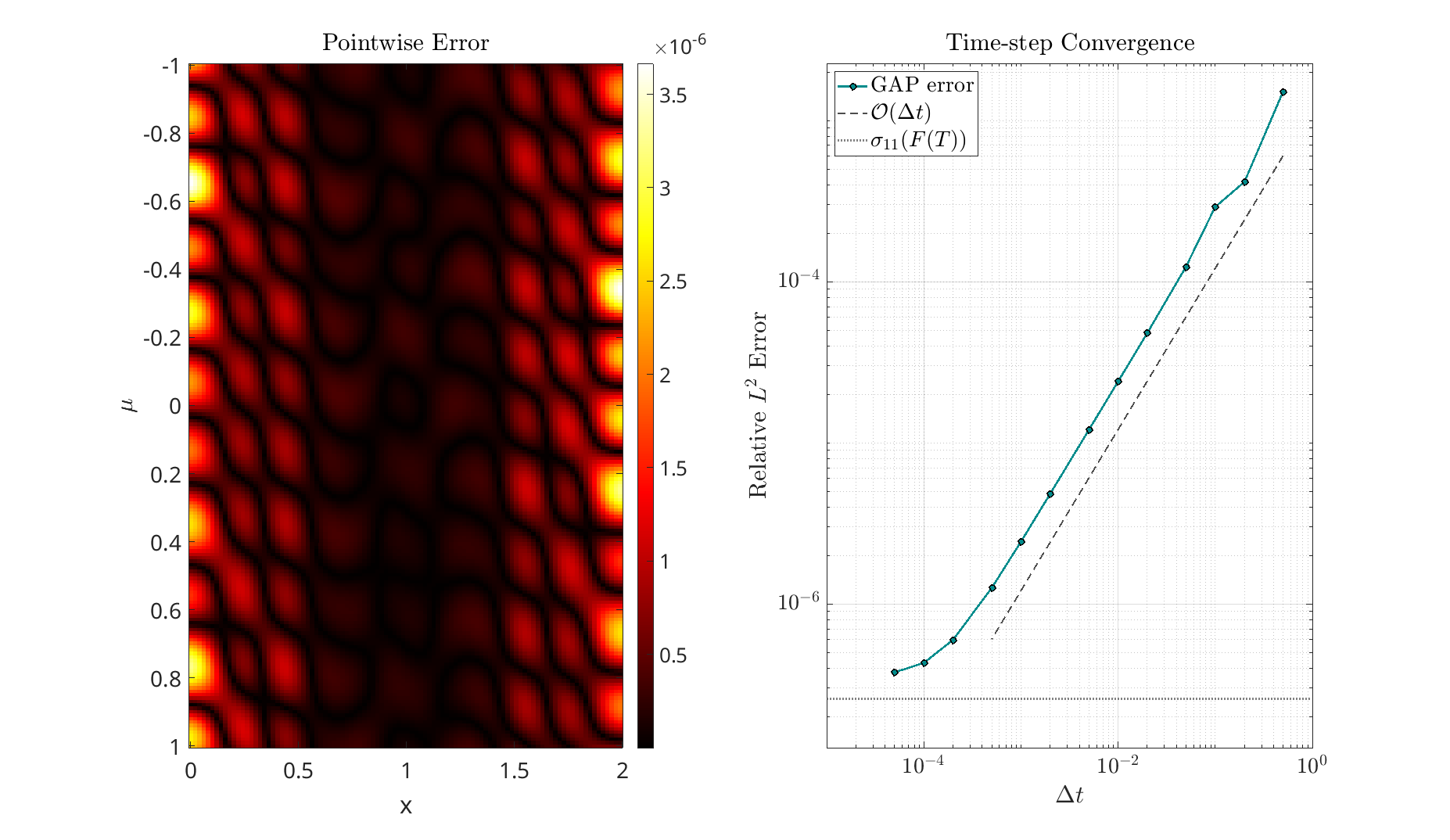}
		\caption{
			Left: Pointwise absolute error between the reference solution and the low-rank approximation at final time \( T = 1 \) for the scaled radiative transfer equation. 
			The low-rank solution is computed with rank \( r = 10 \) and time step \( \Delta t \approx 10^{-4} \), using \texttt{expmv} for time integration. 
			Right: Convergence of the relative \( L^2 \) error as a function of the time step size \( \Delta t \), illustrating first-order accuracy of the scheme in the kinetic regime. 
			The dashed line shows the reference slope \( \mathcal{O}(\Delta t) \), while the horizontal line indicates the saturation level associated with the \((r+1)\)th singular value \( \sigma_{11}(F(T)) \) of the reference solution.}
	\end{figure}

	The right panel of Figure~\ref{fig:pointwiseErrorConvergence} shows that the relative \( L^2 \) error scales linearly with \( \Delta t \), as expected from the first-order accuracy of the GAP scheme, until it saturates at the level of the \((r+1)\)th singular value \( \sigma_{11}(F(T)) \). 
	Here, \( F(T) \) denotes the reference solution of \eqref{eq:discRTE} at final time computed with the same numerical parameters and a time step \( \Delta t = 0.01 \) using numerically exact time-integration via \texttt{expmv} in \textsc{Matlab}.
	This saturation reflects the intrinsic low-rank approximation error: the contribution of neglected singular modes beyond rank \(r\).
	The left panel visualizes the pointwise error at final time \(T = 1\) for \(\Delta t \approx 10^{-4}\). The pointwise error plots show that the discrepancy between the low-rank approximation and the reference solution remains uniformly small across the domain.

\section*{Acknowledgments}
Part of this work was carried out during the workshop \emph{Computational Challenges and Optimization in Kinetic Plasma Physics}, held at the Institute for Mathematical and Statistical Innovation (IMSI), University of Chicago, USA. The support and hospitality of IMSI are gratefully acknowledged. This work has been carried out within the framework of the EUROfusion Consortium, funded by the European Union via the Euratom Research and Training Programme (grant agreement No 101052200 EUROfusion). Views and opinions expressed are however those of the authors only and do not necessarily reflect those of the European Union or the European Commission. The work has been supported by the French Federation for Magnetic Fusion Studies (FR- FCM).

\bibliographystyle{siamplain}
\bibliography{references}
\end{document}